\documentclass[11pt]{article}
\usepackage{t1enc}
\usepackage[latin1]{inputenc}
\usepackage[english]{babel}
\usepackage{amsmath,amsthm}
\usepackage{amsfonts}
\usepackage{latexsym}
\usepackage[dvips]{graphicx}
\usepackage{graphicx}
\usepackage{color}

\usepackage{amsmath}
\usepackage{amssymb}
\usepackage{amsbsy}

\usepackage{amsfonts}

\usepackage{amsfonts,srcltx,mathrsfs}

\title{On extremal cacti with respect to the edge Szeged index\\ and edge-vertex Szeged index}
\author{ Shengjie  He$^1$, Rong-Xia Hao$^1$\footnote{Corresponding author.
Emails: he1046436120@126.com (Shengjie  He), rxhao@bjtu.edu.cn (Rong-Xia Hao), yuaimeimath@163.com (Aimei Yu)}, Aimei Yu$^1$\\
{\small\em 1. Department of Mathematics, Beijing Jiaotong University, Beijing,
100044, China}\\
  }

\date{} \textwidth 16cm \textheight 22cm \topmargin 0 cm \hoffset
-1.5 cm \voffset 0cm

\newtheorem{theorem}{Theorem}[section]
\newtheorem{lemma}[theorem]{Lemma}

\newtheorem{definition}[theorem]{Definition}
\newtheorem{corollary}[theorem]{Corollary}

\begin{document}
\baselineskip 0.50cm \maketitle

\begin{abstract}
The edge Szeged index and edge-vertex Szeged index of a graph are defined as $Sz_{e}(G)=\sum\limits_{uv\in E(G)}m_{u}(uv|G)m_{v}(uv|G)$
and $Sz_{ev}(G)=\frac{1}{2} \sum\limits_{uv \in E(G)}[n_{u}(uv|G)m_{v}(uv|G)+n_{v}(uv|G)m_{u}(uv|G)],$ respectively, where $m_{u}(uv|G)$ (resp., $m_{v}(uv|G)$) is the number of edges whose distance to vertex $u$ (resp., $v$) is smaller than the distance to vertex $v$ (resp., $u$), and $n_{u}(uv|G)$ (resp., $n_{v}(uv|G)$) is the number of vertices whose distance to vertex $u$ (resp., $v$) is smaller than the distance to vertex $v$ (resp., $u$), respectively. A cactus is a graph in which any two cycles have at most one common vertex.
In this paper, the lower bounds of edge Szeged index and edge-vertex Szeged index for cacti with order $n$ and $k$ cycles are determined, and all the graphs that achieve the lower bounds are identified.

{\bf Keywords}: Edge Szeged index; Edge-vertex Szeged index; Cactus.

{\bf 2010 MSC}: 05C40, 05C90
\end{abstract}

\section{Introduction}
The topological indices are quantity values closely related to chemical structure which can be used in theoretical chemistry for understand the physicochemical properties of chemical compounds.
In this paper, we consider two topological indices named the edge Szeged index and edge-vertex Szeged index, which are closely related to two other topological indices, the Wiener index and the Szeged index.

Let $G$ be a connected graph with vertex set $V(G)$ and edge set $E(G)$. For a vertex $u \in V(G)$, the degree of $u$, denote by $d_{G}(u)$, is the number of vertices
which are adjacent to $u$. Call a vertex $u$ a pendant vertex of $G$, if $d_{G}(u)=1$ and call an edge $uv$ a pendant edge of $G$, if $d_{G}(u)=1$ or $d_{G}(v)=1$.
An edge $e$ is called a cut edge of a connected graph $G$ if $G-e$ is disconnect. For any two vertices $u, v \in V(G)$, let $d_{G}(u, v)$ denote the distance between $u$ and $v$ in $G$. Denote by $P_n$, $S_n$ and $C_n$ a path, star and cycle on $n$ vertices, respectively.

A cactus is a graph that any block is either a cut edge or a cycle. It is also a graph in which any two cycles have at most one common vertex. A cycle in a cactus is called end-block if all but one vertex of this cycle have degree 2. If all the cycles in a cactus have exactly one common vertex, then they form a bundle. Let $\mathcal{C}(n,k)$ be the class of all cacti of order $n$ with $k$ cycles.
Let $C_{0}(n,k) \in \mathcal{C}(n,k)$ be a bundle of $k$ triangles with $n-2k-1$ pendant edges attached at the common vertex of the $k$ triangles (see Fig. 1).
The Wiener index is one of the oldest and the most thoroughly studied topological indices. The Wiener index of a graph $G$ is defined as
$$W(G)=\sum\limits_{\{ u, v\} \subseteq V(G) } d_{G}(u, v).$$
This topological index has been stuided extensively and has been found applications in modelling physicochemical properties. The upper and lower bounds and other aspects of the Wiener index of many graphs have been fully studied; see, e.g., \cite{Al.M.J,C.LXL,Deng.H,Do.RI,Dong.H,guts,LinZ,LXue.M,Zhang.X}.

For any edge $e=uv$ of $G$, $V(G)$ can be partitioned into three sets by comparing with the distance of the vertex in $V(G)$ to $u$ and $v$, and the three sets are as follows:
$$N_{u}(e|G)=\{ w\in V(G): d_{G}(u, w) < d_{G}(v, w) \},$$
$$N_{v}(e|G)=\{ w\in V(G): d_{G}(v, w) < d_{G}(u, w) \},$$
$$N_{0}(e|G)=\{ w\in V(G): d_{G}(u, w) = d_{G}(v, w) \}.$$
The number of vertices of $N_{u}(e|G)$, $N_{v}(e|G)$, and $N_{0}(e|G)$ are denoted by $n_{u}(e|G)$, $n_{v}(e|G)$ and $n_{0}(e|G)$, respectively. If $G$ is a tree, then the formula $W(G)=\sum\limits_{e=uv \in E(G)}n_{u}(e|G)n_{v}(e|G)$
gives a long time known property of the Wiener index.

According to the above formula and result, a new topological index, named by Szeged index, was introduced by Gutman \cite{Gut.A}, which is an extension of the Wiener index and defined by
$$Sz(G)=\sum\limits_{e=uv \in E(G)}n_{u}(e|G)n_{v}(e|G).$$
Since it has been proved to be of great applications in the study of the modeling physicochemical properties of chemical compounds and drugs,
the Szeged index has been studied extensively by many researchers, see, e.g., \cite{AF.T,Ao.M,C.LXL,Do.A,Kh.P,P.V.K,Kha.M.H,Lei.H,Wang,zhang.H,ZhouB.X}.

If $e=uv$ is an edge of $G$ and $w$ is a vertex of $G$, then the distance between $e$ and $w$ is defined as $d_{G}(e,w) = {\rm{min}} \{ d_{G}(u,w),d_{G}(v,w) \}$. For $e=uv \in E(G)$, let $M_u(e|G)$ be the set of edges whose distance to the vertex $u$ is smaller than the distance to the vertex $v$, and $M_v(e|G)$ be the set of edges whose distance to the vertex $v$ is smaller than the distance to the vertex $u$. Set $m_u(e|G)=|M_u(e|G)|$ and $m_v(e|G)=|M_v(e|G)|$. Gutman and Ashrafi \cite{Gut.A.R} introduced an edge version of the Szeged index, named edge Szeged index. The edge Szeged index of $G$ is defined as
$$Sz_{e}(G)=\sum\limits_{uv \in E(G)}m_{u}(uv|G)m_{v}(uv|G).$$
The edge-vertex Szeged index \cite{Fag.M} of $G$ is defined as:
$$Sz_{ev}(G)=\frac{1}{2} \sum\limits_{uv \in E(G)} [n_{u}(uv|G)m_{v}(uv|G)+n_{v}(uv|G)m_{u}(uv|G)].$$

In \cite{Gut.A.R}, some basic properties of the edge Szeged index were established by Gutman. It can be checked that the pendant edges make no contributions to the edge Szeged index of a graph.
In \cite{Kh.M.H}, the edge Szeged index of the Cartesian product of graphs was computed.
In \cite{Cai.X}, Cai and Zhou determined the $n$-vertex unicyclic graphs with the largest, the second largest, the smallest and the second smallest edge Szeged indices, respectively. In \cite{Wang}, Wang determined a lower bound of the Szeged index for cacti of order $n$ with $k$ cycles. In \cite{Al.M.J}, Alaeiyan and Asadpour
characterized the edge-vertex Szeged index of bridge graphs.

In this paper, we give the lower bounds of the edge Szeged index and edge-vertex Szeged index for cacti of order $n$ with $k$ cycles, and also characterize those graphs that achieve the lower bounds. Note that $\mathcal{C}(3, 0)=\{P_3\}$, $\mathcal{C}(3, 1)=\{C_3\}$, $\mathcal{C}(4, 0)=\{P_4,S_4\}$ and $\mathcal{C}(4, 1)=\{C_4,C_{0}(4, 1)\}$. 
By simple computations, among the graphs in $\mathcal{C}(4, 0)$, $S_4$ is the graph with minimum edge Szeged index and edge-vertex Szeged index, respectively. For
the graphs in $\mathcal{C}(4, 1)$, $Sz_{e}(C_4)<Sz_{e}(C_{0}(4, 1))$ and $Sz_{ev}(C_4)>Sz_{ev}(C_{0}(4, 1))$. So in this paper, we only consider the graphs with order $n\geq 5$. Our main result is the following Theorem 1.1.

\begin{theorem}\label{T3.1} For any $G \in \mathcal{C}(n, k)$ ($n\geq 5$), we have
\begin{eqnarray*}Sz_{e}(G) &\geq &2kn+2k^{2}-5k,\\
Sz_{ev}(G) &\geq &\frac{1}{2} (n^{2}-3n+3kn-5k+2),\end{eqnarray*}
with equalities if and only if $G \cong C_{0}(n, k)$.
\end{theorem}

It is obvious that $2kn+2k^{2}-5k$ and $3kn-5k$ get the minimum values when $k=1$. Thus we have the following corollary.

\begin{corollary}\label{T3.2} For any cactus $G \in \mathcal{C}(n,k)$ with $n\geq 5$ and $k\neq 0$, we have
\begin{eqnarray*}Sz_{e}(G) &\geq& 2n-3,\\
Sz_{ev}(G) &\geq& \frac{1}{2}(n^{2}-3),\end{eqnarray*}
with equalities if and only if $G \cong C_{0}(n, 1)$.
\end{corollary}

\vspace{0.8cm}

\begin{center}   \setlength{\unitlength}{0.7mm}
\begin{picture}(30,40)

\put(-15,10){\circle*{1}}
\put(-5,10){\circle*{1}}
\put(5,10){\circle*{1}}
\put(30,10){\circle*{1}}

\put(15,30){\circle*{1}}

\put(-5,50){\circle*{1}}
\put(-15,50){\circle*{1}}

\put(0,50){\circle*{1}}
\put(10,50){\circle*{1}}

\put(35,50){\circle*{1}}
\put(25,50){\circle*{1}}

\put(-15,50){\line(1,0){10}}
\put(0,50){\line(1,0){10}}
\put(25,50){\line(1,0){10}}

\put(15,30){\line(-3,-2){30}}
\put(15,30){\line(-1,-1){20}}
\put(15,30){\line(-1,-2){10}}
\put(15,30){\line(3,-4){15}}

\put(15,30){\line(-3,2){30}}
\put(15,30){\line(-1,1){20}}
\put(15,30){\line(-3,4){15}}
\put(15,30){\line(-1,4){5}}

\put(15,30){\line(1,2){10}}
\put(15,30){\line(1,1){20}}

\put(11,8){$\cdots$}
\put(16,8){$\cdots$}
\put(14,48.5){$\cdots$}

\put(-22,4){$n-2k-1$ pendant vertices}
\put(-3,52){$k$ triangles}

\put(-15,-5){Fig. 1. $C_{0}(n, k)$}

\end{picture} \end{center}

\vspace{0.1cm}

The rest of this paper is organized as follows. In Section 2, we establish two useful lemmas. In Section 3, we present some transformations of graphs, and use these transformations to prove Theorem \ref{T3.1}. Section 4 is a conclusion.

\section{Useful lemmas}

In this section, we will introduce two useful lemmas which will be used frequently in next section. For short, for an edge $e=uv$ of graph  $G$, we denote $$m'_u(e|G)=n_{u}(e|G)m_{v}(e|G) \mbox{ and } m'_v(e|G)=n_{v}(e|G)m_{u}(e|G).$$

\begin{center}   \setlength{\unitlength}{0.7mm}
\begin{picture}(30,40)

\put(-5,30){\circle*{1}}
\put(-15,30){\circle{20}}
\put(5,30){\circle{20}}
\put(-8,5){$G$}
\put(52,5){$G'$}

\put(-11,30){$u$}
\put(49,30){$u$}

\put(55,30){\circle*{1}}

\put(-17,15){$G_{0}$}
\put(3,15){$G_{1}$}
\put(43,15){$G_{0}$}
\put(63,15){$G_{2}$}

\put(45,30){\circle{20}}
\put(65,30){\circle{20}}

\put(-15,-5){Fig. 2. $G$ and $G'$ in Lemma 2.1}
\end{picture} \end{center}

\begin{lemma}\label{lem2.1}
Let $G$ and $G'$ be the graphs shown as in $Fig. \, 2$, where $G$ consists of $G_0$ and $G_1$ with a common vertex $u$, and $G'$ consists of $G_0$ and $G_2$ with a common vertex $u$. Then each of the followings holds:\\
(i) For any edge $e=w_1w_2\in E(G_0)$ and $1\leq i\leq 2$, we have \begin{eqnarray*}n_{w_i}(e|G)&=&n_{w_i}(e|G_0)+\delta(u)(|V(G_1)|-1),\\
                 m_{w_i}(e|G)&=&m_{w_i}(e|G_0)+\delta(u)|E(G_1)|,\end{eqnarray*}
                 where $$ \delta(u)=\left\{
                                                 \begin{array}{ll}
                                                   1, & \hbox{$u\in N_{w_i}(e|G_0)$ ;} \\
                                                   0, & \hbox{otherwise.}
                                                 \end{array}
                                               \right.
  $$
$(ii)$ If $|V(G_1)|=|V(G_2)|$ and $|E(G_1)|=|E(G_2)|$, then
\begin{eqnarray*}
\sum\limits_{e=w_1w_2 \in E(G_0)} m_{w_1}(e|G)m_{w_2}(e|G)&=&\sum\limits_{e=w_1w_2 \in E(G_0)} m_{w_1}(e|G')m_{w_2}(e|G'),\\
\sum\limits_{e=w_1w_2 \in E(G_0)} [m'_{w_1}(e|G)+m'_{w_2}(e|G)]&=&\sum\limits_{e=w_1w_2 \in E(G_0)}[ m'_{w_1}(e|G')+m'_{w_2}(e|G')].
\end{eqnarray*}
\end{lemma}

\begin{proof}
For any edge $e=w_1w_2 \in E(G_0)$ and any vertex $w\in V(G_1)$, $d_{G}(w_i,w)=d_{G_0}(w_i,u)+d_{G_1}(u,w),$ which implies Lemma \ref{lem2.1}(i).

As $|V(G_1)|=|V(G_2)|$ and $|E(G_1)|=|E(G_2)|$, for any edge $e=w_1w_2 \in E(G_0)$ and $1\leq i\leq 2$, by Lemma \ref{lem2.1}(i), $n_{w_i}(e|G)=n_{w_i}(e|G')$ and $m_{w_i}(e|G)=m_{w_i}(e|G')$. Hence Lemma \ref{lem2.1}(ii) holds.
\end{proof}

\begin{definition}\label{Def-1}Let $G$ be a graph of order $n$ with a cycle $C_l=v_1v_2 \cdots v_{l}v_1$.
Assume that $G-E(C_{l})$ has exactly $l$ components $G_1,G_2,\cdots, G_l$, where $G_i$ is the component of $G-E(C_{l})$ that contains $v_i$ for $1\leq i\leq l$.
For $1\leq i \leq l$, let $n_i=|V(G_i)|$ and $m_i=|E(G_i)|$. Let $l=2k$ or $2k+1$. Set $m=\sum_{i=1}^{l}m_i$ and \begin{eqnarray*}y_i&=&m_{i}+m_{i-1}+ \cdots +m_{i-k+1},\\
x_i&=&n_{i}+n_{i-1}+ \cdots +n_{i-k+1},\end{eqnarray*} where the subscripts are taken modulo $l$. 
\end{definition}

\begin{lemma}\label{lem2.2}
Let $G$, $x_i$, $y_i$, $n_i$ and $m_i$ be defined as in Definition \ref{Def-1}. For $1\leq i \leq l$, denote $e_i=v_iv_{i+1}$, where the subscripts are taken modulo $l$. Let $$
f(m,l)=
\left\{
\begin{array}{ll}
2k(k-1)(m+k-1), & \hbox{$l=2k$,} \\
k^{2}(2m+2k+1), & \hbox{$l=2k+1$,}
\end{array}
\right.
$$and $$
g(n,m,l)=
\left\{
\begin{array}{ll}
2k^2(n+m)-2kn, & \hbox{$l=2k$,} \\
2k^2(n+m), & \hbox{$l=2k+1$.}
\end{array}
\right.
$$Then each of the followings holds:
\begin{itemize}
\item[(i)] $
\sum\limits_{e_i \in E(C_{l})}m_{v_i}(e_i|G)m_{v_{i+1}}(e_i|G)  =
\left\{
\begin{array}{ll}
f(m,l)+\sum\limits_{i=1}^{2k}y_i(m-y_i), & \hbox{$l=2k$;} \\
f(m,l)+\sum\limits_{i=1}^{2k+1}y_i(m-m_{i-k}-y_i), & \hbox{$l=2k+1$.}
\end{array}
\right.
$
\item[(ii)] $
\sum\limits_{e_i \in E(C_{l})}[m'_{v_i}(e_i|G)+m'_{v_{i+1}}(e_i|G)]  \\
=\left\{
\begin{array}{ll}
g(n,m,l)+\sum\limits_{i=1}^{2k}[(x_i-k) (m-y_i)+y_i(n-x_i-k)], & \hbox{$l=2k$;} \\
g(n,m,l)+\sum\limits_{i=1}^{2k+1}[(x_i-k)(m-m_{i-k}-y_i)+y_i(n-n_{i-k}-x_i-k)], & \hbox{$l=2k+1$.}
\end{array}
\right.
$

\item[(iii)]  Moreover, we have
\begin{eqnarray}\sum\limits_{e_i \in E(C_{l})}m_{v_i}(e_i|G)m_{v_{i+1}}(e_i|G)\geq  f(m,l)\label{E-11} \\\sum\limits_{e_i \in E(C_{l})}[m'_{v_i}(e_i|G)+m'_{v_{i+1}}(e_i|G)]\geq  g(n,m,l),\label{E-12}\end{eqnarray}
where equalities hold if and only if there is at most one positive integer among $m_1,m_2,\cdots,m_l$.
\end{itemize}

\end{lemma}

\begin{proof}
Let $G$ be a graph defined as in Definition \ref{Def-1}.
By Lemma~\ref{lem2.1}(i), for any $e\in E(C_l)$ and one end $v$ of $e$, we have
\begin{eqnarray*}
n_{v}(e|G)=\sum_{j\in N}n_j\mbox{ ~~~~ and ~~~~ }
m_{v}(e|G)=m_{v}(e|C_l)+\sum_{j\in N}m_j,
\end{eqnarray*}where $N=\{j: v_j\in N_{v}(e|G)\}$. So we have\begin{eqnarray}
m_{v_{i}}(e_i|G)m_{v_{i+1}}(e_i|G)  =
\left\{
\begin{array}{ll}
(y_i+k-1)(m-y_i+k-1), & \hbox{$l=2k$,} \\
(y_i+k)(m-m_{i-k}-y_i+k), & \hbox{$l=2k+1$,}
\end{array}
\right.\label{E-9}
\end{eqnarray}
and \begin{eqnarray}
m'_{v_i}(e_i|G)+m'_{v_{i+1}}(e_i|G)=
 \left\{
\begin{array}{ll}
x_i(m-y_i+k-1)+(n-x_i)(y_i+k-1), & \hbox{$l=2k$,} \\
x_i(m-m_{i-k}-y_i+k)+(n-n_{i-k}-x_i)(y_i+k), & \hbox{$l=2k+1$.}
\end{array}\right.\label{E-10}
\end{eqnarray}

By Equality (\ref{E-9}), if $l=2k$,
\begin{eqnarray*}
\sum\limits_{e_i \in E(C_{l})}m_{v_i}(e_i|G)m_{v_{i+1}}(e_i|G)&=&\sum\limits_{i=1}^{2k}(y_i+k-1)(m-y_i+k-1)\\
&=&\sum\limits_{i=1}^{2k}\left[m(k-1)+(k-1)^{2}+y_i(m-y_{i})\right]\\
                 &=&2k(k-1)(m+k-1)+\sum\limits_{i=1}^{2k}y_i(m-y_i);
\end{eqnarray*}
if $l=2k+1$,
\begin{eqnarray*}
\sum\limits_{e_i \in E(C_{l})}m_{v_i}(e_i|G)m_{v_{i+1}}(e_i|G) &=&
\sum\limits_{i=1}^{2k+1}(y_i+k)(m-m_{i-k}-y_i+k)\\
&=&k(2k+1)(m+k)+\sum_{i=1}^{2k+1}y_{i}(m-m_{i-k}-y_{i})-\sum_{i=1}^{2k+1}km_{i-k}\\
&=&k(2k+1)(m+k)+\sum_{i=1}^{2k+1}y_{i}(m-m_{i-k}-y_{i})-km\\
&=&k^{2}(2m+2k+1)+\sum_{i=1}^{2k+1}y_{i}(m-m_{i-k}-y_{i}).
\end{eqnarray*}
This justifies Lemma \ref{lem2.2}(i).

By Equality (\ref{E-10}), if $l=2k$,
\begin{eqnarray*}
\sum\limits_{e_i \in E(C_{l})}[m'_{v_i}(e_i|G)+m'_{v_{i+1}}(e_i|G)]
&=&\sum\limits_{i=1}^{2k}\left[x_i(m-y_i+k-1)+(n-x_i)(y_i+k-1)\right]\\
&=&\sum\limits_{i=1}^{2k}[n(k-1)+mk+m(x_i-k)+ny_i-2x_iy_i]\\
&=&2k^2(n+m)-2kn+\sum\limits_{i=1}^{2k}[(x_i-k) (m-y_i)+y_i(n-x_i-k)];
\end{eqnarray*}
if $l=2k+1$,
\begin{eqnarray*}
&&\sum\limits_{e_i \in E(C_{l})}[m'_{v_i}(e_i|G)+m'_{v_{i+1}}(e_i|G)]\\
&=&\sum\limits_{i=1}^{2k+1}\left[x_i(m-m_{i-k}-y_i+k)+(n-n_{i-k}-x_i)(y_i+k)\right]\\
&=&\sum\limits_{i=1}^{2k+1}[x_i(m-m_{i-k}-y_i)+y_i(n-n_{i-k}-x_{i})-kn_{i-k}+kn]\\
&=&\sum\limits_{i=1}^{2k+1}[(k+x_i-k)(m-m_{i-k}-y_i)+y_i(n-n_{i-k}-x_{i})-kn_{i-k}+kn]\\
&=&\sum\limits_{i=1}^{2k+1}[(x_i-k)(m-m_{i-k}-y_i)+y_i(n-n_{i-k}-x_i-k)+k(m+n-m_{i-k}-n_{i-k})]\\
&=&\sum\limits_{i=1}^{2k+1}[(x_i-k)(m-m_{i-k}-y_i)+y_i(n-n_{i-k}-x_i-k)]-\sum\limits_{i=1}^{2k+1}k(m_{i-k}+n_{i-k})\\
&&+k(n+m)(2k+1)\\
&=&\sum\limits_{i=1}^{2k+1}[(x_i-k)(m-m_{i-k}-y_i)+y_i(n-n_{i-k}-x_i-k)]+2k^2(n+m).
\end{eqnarray*}
This justifies Lemma \ref{lem2.2}(ii).

Let $1\leq i\leq l$. By definitions, $y_i$, $m-y_i$ and $x_i-k$ are nonnegative integers; if $l=2k$, $n-x_i-k\geq 0$; if $l=2k+1$, $m-m_{i-k}-y_i\geq 0$ and $n-n_{i-k}-x_i-k\geq 0$. Moreover, $x_{i}-k=0$ if and only if $y_i=0$; if $l=2k$,  then $n-x_{i}-k=0$ if and only if $m-y_i=0$; if $l=2k+1$, $n-n_{i-k}-x_{i}-k=0$ if and only if $m-m_{i-k}-y_i=0$.
Hence by Lemma \ref{lem2.2}(i) and (ii), Inequalities (\ref{E-11}) and (\ref{E-12}) hold. Furthermore, the equality in (\ref{E-12}) holds if and only if the equality in (\ref{E-11}) holds.

If there is at most one positive integer among $m_1,m_2,\cdots,m_l$, say $m_1\geq0$ and $m_i=0$ for $2\leq i\leq l$, then $m-y_i=0$ (or $m-m_{i-k}-y_i=0$) if $1\leq i\leq k$, and $y_i=0$ if $k+1\leq i\leq l$. Then the equality in (\ref{E-11}) holds. Without loss of generality, now we assume that $m_1>0$ and $m_j>0$. If $l=2k$ or $2k+1$, by symmetry, assume that $1<j\leq k+1$. Then $y_1(m-y_1)>0$ and the equality in (\ref{E-11}) does not hold. This completes the proof of Lemma \ref{lem2.2}.\end{proof}

\section{Cacti with minimum edge Szeged index and edge-vertex Szeged index in $\mathcal{C}(n, k)$}

In this section, we characterize several transformations of cacti which keep the order and the number of edges of the cacti, but decrease the edge Szeged index and edge-vertex Szeged index of the cacti. Using these transformations, we determine the lower bounds of the edge Szeged index and edge-vertex Szeged index of $\mathcal{C}(n,k)$ and those graphs that achieve the lower bounds.

\begin{center}   \setlength{\unitlength}{0.7mm}
\begin{picture}(30,40)

\put(-15,30){\circle*{1}}
\put(-5,30){\circle*{1}}
\put(-15,30){\line(1,0){10}}

\put(-25,30){\circle{20}}
\put(5,30){\circle{20}}

\put(-21,30){$u_{1}$}
\put(-4,30){$u_{2}$}

\put(-13,5){$G$}
\put(52,5){$G'$}

\put(55,30){\circle*{1}}
\put(55,20){\circle*{1}}

\put(49,30){$u_{1}$}
\put(52,15){$u_{2}$}

\put(45,30){\circle{20}}
\put(65,30){\circle{20}}

\put(55,30){\line(0,-1){10}}

\put(-15,-5){Fig. 3. $G$ and $G'$ in Lemma 3.1}

\end{picture} \end{center}

\begin{lemma}\label{lem3.1}
Let $G$ be a graph with order $n$ and a cut edge $u_{1}u_{2}$, and $G'$ be the graph obtained from $G$ by contracting the edge $u_1u_2$ and attaching a pendant edge (which is also denoted by $u_1u_2$) at the contracting vertex; see $Fig. \,3$. If $d_{G}(u_{i}) \geq 2$ for $i=1, 2$, we have  $Sz_{e}(G') < Sz_{e}(G)$ and $Sz_{ev}(G') < Sz_{ev}(G)$.
\end{lemma}
\begin{proof}

Let $G_1$ and $G_2$ be the components of $G-u_{1}u_{2}$ that contain $u_1$ and $u_2$, respectively. By Lemma~\ref{lem2.1}(ii), we have \begin{eqnarray*}\sum\limits_{e=uv \in E(G_1)\cup E(G_2)} m_u(e|G)m_v(e|G)&=&\sum\limits_{e=uv \in E(G_1)\cup E(G_2)} m_u(e|G')m_v(e|G'),\\
\sum\limits_{e=uv \in E(G_1)\cup E(G_2)} m'_u(e|G)+m'_v(e|G)&=&\sum\limits_{e=uv \in E(G_1)\cup E(G_2)} m'_u(e|G')+m'_v(e|G').\end{eqnarray*}
Hence by definitions and Lemma \ref{lem2.1}(ii),  we have
\begin{eqnarray*}
Sz_e(G)-Sz_e(G') & =&m_{u_1}(u_1u_2|G)m_{u_2}(u_1u_2|G)
                    -m_{u_1}(u_1u_2|G')m_{u_2}(u_1u_2|G')\\
                 & =&|E(G_1)||E(G_2)|-0,\\
2(Sz_{ev}(G)-Sz_{ev}(G'))&=&m'_{u_1}(u_1u_2|G)+m'_{u_2}(u_1u_2|G)
                    -m'_{u_1}(u_1u_2|G')-m'_{u_2}(u_1u_2|G')\\
& =&(|V(G_1)||E(G_2)|+|V(G_2)||E(G_1)|)-0- (|E(G_{1})|+|E(G_{2})|)\\
                 & =&(|V(G_1)|-1)|E(G_2)|+(|V(G_2)|-1)|E(G_1)|.
                 \end{eqnarray*}
As  $d_G(u_i)\geq 2$ for $1\leq i\leq 2$, we have $|V(G_{i})|\geq 2$ and $|E(G_i)|\geq 1$. Hence $Sz_{e}(G') < Sz_{e}(G)$ and $Sz_{ev}(G') < Sz_{ev}(G)$.
\end{proof}

\begin{lemma}\label{lem3.2}Let $G$ be a graph of order $n$ with a cycle $C_l=v_1v_2 \cdots v_{l}v_1$.
Assume that $G-E(C_{l})$ has exactly $l$ components $G_1,G_2,\cdots, G_l$, where $G_i$ is the component of $G-E(C_{l})$ that contains $v_i$ for $1\leq i\leq l$.
Let
 \begin{eqnarray*}G'=G- \displaystyle\cup_{i=2}^l\{wv_i: w\in N_{G_i}(v_i)\}+ \cup_{i=2}^l\{wv_1:  w\in N_{G_i}(v_i)\}.\end{eqnarray*}
Then $Sz_e(G') \leq Sz_e(G)$ and $Sz_{ev}(G') \leq Sz_{ev}(G)$ with equalities if and only if $C_{l}$ is an end-block, that is, $G\cong G'$.
\end{lemma}
\begin{proof}
By Lemma~\ref{lem2.1}(ii), we have \begin{eqnarray*}\sum\limits_{e=uv \not\in E(C_l)} m_u(e|G)m_v(e|G)&=&\sum\limits_{e=uv \not\in E(C_l)} m_u(e|G')m_v(e|G'),\\
\sum\limits_{e=uv \not\in E(C_l)} [m'_u(e|G)+m'_v(e|G)]&=&\sum\limits_{e=uv \not\in E(C_l)} [m'_u(e|G')+m'_v(e|G')].\end{eqnarray*}
So \begin{eqnarray*}Sz_e(G)-Sz_e(G')=\sum\limits_{e=uv \in E(C_l)} m_u(e|G)m_{v}(e|G)-\sum\limits_{e=uv \in E(C_l)} m_{u}(e|G')m_{v}(e|G'),\label{E-1}\end{eqnarray*}
\begin{eqnarray*}2(Sz_{ev}(G)-Sz_{ev}(G'))=\sum\limits_{e=uv \in E(C_l)} [m'_{u}(e|G)+m'_{v}(e|G)]-\sum\limits_{e=uv \in E(C_l)} [m'_{u}(e|G')+m'_{v}(e|G')].\label{E-2}\end{eqnarray*}
Then by Lemma \ref{lem2.2}(iii), Lemma \ref{lem3.2} holds immediately .\end{proof}
\vspace{-0.8cm}
\begin{center}   \setlength{\unitlength}{0.7mm}
\begin{picture}(30,50)

\put(-30,35){\circle*{1}}
\put(-20,40){\circle*{1}}
\put(-10,40){\circle*{1}}
\put(-5,30){\circle*{1}}
\put(-20,20){\circle*{1}}
\put(-10,20){\circle*{1}}
\put(-30,25){\circle*{1}}
\put(-5,30){\line(-1,2){5}}
\put(-5,30){\line(-1,-2){5}}
\put(-10,40){\line(-1,0){10}}
\put(-10,20){\line(-1,0){10}}
\put(-20,40){\line(-2,-1){10}}
\put(-20,20){\line(-2,1){10}}
\put(5,30){\circle{20}}

\put(30,35){\circle*{1}}
\put(40,40){\circle*{1}}
\put(50,40){\circle*{1}}
\put(55,30){\circle*{1}}
\put(40,20){\circle*{1}}
\put(50,20){\circle*{1}}
\put(30,25){\circle*{1}}
\put(55,30){\line(-1,2){5}}
\put(55,30){\line(-1,-2){5}}
\put(55,30){\line(-3,-2){15}}
\put(55,30){\line(-3,2){15}}
\put(40,40){\line(-2,-1){10}}
\put(40,20){\line(-2,1){10}}
\put(65,30){\circle{20}}

\put(56,30){$v_{1}$}
\put(-4,30){$v_{1}$}

\put(-11,15){$v_{2}$}
\put(-21,15){$v_{3}$}

\put(49,15){$v_{2}$}
\put(39,15){$v_{3}$}

\put(-11,42){$v_{r}$}
\put(-24,42){$v_{r-1}$}

\put(49,42){$v_{r}$}
\put(36,42){$v_{r-1}$}

\put(-31,28){$\vdots$}

\put(29,28){$\vdots$}

\put(-13,5){$G$}
\put(52,5){$G'$}

\put(55,30){\circle*{1}}
\put(65,30){\circle{20}}
\put(-15,-5){Fig. 4. $G$ and $G'$ in Lemma 3.3}

\end{picture} \end{center}

\begin{lemma}\label{lem3.3}
Let $G$ be a graph of order $n$ with an end-block $C_r=v_{1}v_{2} \cdots v_{r}v_{1}$  $(r \geq 5)$, and $G'=G-\{ v_2v_3, v_{r-1}v_{r} \}+\{ v_1v_3, v_1v_{r-1} \}$; see $Fig. \,4$. If $d_{G}(v_{1}) \geq 2$, we have $Sz_{e}(G') < Sz_{e}(G)$ and $Sz_{ev}(G') < Sz_{ev}(G)$.

\end{lemma}

\begin{proof}
Let $E'=E(C_{r})-\{ v_{2}v_{3}, v_{r-1}v_{r} \}+ \{v_{1}v_{3}, v_{1}v_{r-1}\}$ (where $E' \subseteq E(G')$) and $m=|E(G) \backslash E(C_{r})|$.
By Lemma~\ref{lem2.1}(ii), we have
\begin{eqnarray}
Sz_{e}(G)-Sz_{e}(G')=\sum\limits_{e=xy \in E(C_{r})}m_x(e|G)m_y(e|G)-\sum\limits_{e=xy \in E'}m_x(e|G')m_y(e|G'),\label{E-3}\end{eqnarray}
\begin{eqnarray}2(Sz_{ev}(G)-Sz_{ev}(G'))=\sum\limits_{e=xy \in E(C_{r})}[m'_x(e|G)+m'_y(e|G)]-\sum\limits_{e=xy \in E'}[m'_x(e|G')+m'_y(e|G')].\label{E-4}
\end{eqnarray}
As $v_1v_j$ $(j\in \{2,r\})$ is a pendant edge of $G'$, we have
\begin{eqnarray}
m_{v_1}(v_1v_j|G')m_{v_j}(v_1v_j|G')=0 \mbox{ and } m'_{v_1}(v_1v_j|G')+m'_{v_j}(v_1v_j|G')=m+r-1.\label{E-5}
\end{eqnarray}

By Lemma~\ref{lem2.2}(iii) and Equalities (\ref{E-3}) and (\ref{E-5}), we have
\begin{eqnarray*} Sz_{e}(G)-Sz_{e}(G')
&=&f(m,r)-f(m+2,r-2)\\
&=&\left\{
\begin{array}{ll}
2k(k-1)(m+k-1)-2(k-1)(k-2)(m+k), & \hbox{$r=2k$,} \\
k^2(2m+2k+1)-(k-1)^2(2m+2k+3), & \hbox{$r=2k+1$.}
 \end{array}
 \right.\\
 &=&\left\{
\begin{array}{ll}
2(k-1)(2m+k), & \hbox{$r=2k$,} \\
2k^2+(4k-3)+2m(2k-1), & \hbox{$r=2k+1$.}
 \end{array}
 \right.
\end{eqnarray*}
Since $r\geq 5$, we have $k\geq 2$, and so  $Sz_{e}(G') < Sz_{e}(G)$.

By Lemma~\ref{lem2.2}(iii) and Equalities (\ref{E-4}) and (\ref{E-5}), we have
\begin{eqnarray*} &&2(Sz_{ev}(G)-Sz_{ev}(G'))\\
&=&g(n,m,r)-g(n,m+2,r-2)-2(m+r-1)\\
&=&\left\{
\begin{array}{ll}
2k^2(m+n)-2kn-2(k-1)^2(m+n+2)+2n(k-1)-2(m+2k-1), & \hbox{$r=2k$ ,} \\
2k^2(m+n)-2(k-1)^2(m+n+2)-2(m+2k), & \hbox{$r=2k+1$.}
\end{array}
\right.\\
&=&\left\{
\begin{array}{ll}
4m(k-1)+2n(k-2)+2k(n-2k)+4k-2, & \hbox{$r=2k$ ,} \\
4m(k-1)+2n(k-1)+2k(n-2k)+4k-4, & \hbox{$r=2k+1$.}
\end{array}
\right.
  \end{eqnarray*}
Thus we have $Sz_{ev}(G') < Sz_{ev}(G)$.
\end{proof}

\begin{center}   \setlength{\unitlength}{0.7mm}
\begin{picture}(30,40)

\put(65,30){\circle{20}}
\put(5,30){\circle{20}}

\put(-4,30){$v_{1}$}
\put(-31.5,30){$v_{3}$}

\put(-22,40){$v_{2}$}
\put(-22,20){$v_{4}$}
\put(38,40){$v_{2}$}
\put(38,20){$v_{4}$}

\put(56,30){$v_{1}$}
\put(52,15){$v_{3}$}

\put(-5,30){\line(-1,1){10}}
\put(-5,30){\line(-1,-1){10}}
\put(-25,30){\line(1,1){10}}
\put(-25,30){\line(1,-1){10}}

\put(55,30){\line(-1,1){10}}
\put(55,30){\line(-1,-1){10}}
\put(45,20){\line(0,1){20}}

\put(-13,5){$G$}
\put(52,5){$G'$}

\put(-15,40){\circle*{1}}
\put(-15,20){\circle*{1}}
\put(-25,30){\circle*{1}}

\put(-5,30){\circle*{1}}

\put(55,30){\circle*{1}}
\put(55,20){\circle*{1}}
\put(45,40){\circle*{1}}
\put(45,20){\circle*{1}}

\put(55,30){\line(0,-1){10}}

\put(-15,-5){Fig. 5. $G$ and $G'$ in Lemma 3.4}

\end{picture} \end{center}

\begin{lemma}\label{lem3.4}
Let $G$ be a graph with order $n $ $(n\geq 5)$ and an end-block $C_4=v_{1}v_{2}v_{3}v_{4}v_{1}$, and $G'=G-\{ v_2v_3, v_3v_4 \}+\{ v_2v_4, v_1v_3 \}$; see $Fig. \,5$. Then we have $Sz_{e}(G') < Sz_{e}(G)$ and $Sz_{ev}(G') < Sz_{ev}(G)$.

\end{lemma}

\begin{proof}
Let $E'= \{ v_{1}v_{2}, v_{1}v_{3}, v_{1}v_{4}, v_{2}v_{4}\} \subseteq E(G')$ and $|E(G)\setminus E(C_{4})|=m$.
By Lemma~\ref{lem2.1}(ii), we have
\begin{eqnarray*}
Sz_{e}(G)-Sz_{e}(G')=\sum\limits_{e=xy \in E(C_{4})}m_x(e|G)m_y(e|G)-\sum\limits_{e=xy \in E'}m_x(e|G')m_y(e|G'),\label{E-6}\end{eqnarray*}
\begin{eqnarray*}2(Sz_{ev}(G)-Sz_{ev}(G'))=\sum\limits_{e=xy \in E(C_{4})}[m'_x(e|G)+m'_y(e|G)]-\sum\limits_{e=xy \in E'}[m'_x(e|G')+m'_y(e|G')].\label{E-7}
\end{eqnarray*}
As $v_1v_3$ is a pendant edge of $G'$, we have
\begin{eqnarray*}
m_{v_1}(v_1v_3|G')m_{v_3}(v_1v_3|G')=0 \mbox{ and } m'_{v_1}(v_1v_3|G')+m'_{v_3}(v_1v_3|G')=3+m.\label{E-8}
\end{eqnarray*}

By Lemma~\ref{lem2.2}(iii), 
we have \begin{eqnarray*}Sz_{e}(G)-Sz_{e}(G')&=&f(m,4)-f(m+1,3)
=4(m+1)-(2m+5)=2m-1,\\
2[Sz_{ev}(G)-Sz_{ev}(G')]&=&g(n,m,4)-g(n,m+1,3)\\
&=&4n+8m-2(n+m+1)-(m+3)\\&=&2n+5m-5.
\end{eqnarray*}
Since $n\geq 5$, $m\geq 1$.
Hence, we have our conclusions.
\end{proof}

{\bf Proof of Theorem \ref{T3.1}.}
Suppose that $G$ is the graph that has minimum edge Szeged index (resp., edge-vertex Szeged index ) in $\mathcal{C}(n, k)$ ($n\geq 5$). Then there is no graph $G'\in \mathcal{C}(n, k)$ such that $Sz_e(G')<Sz_e(G)$ (resp., $Sz_{ev}(G')<Sz_{ev}(G))$. By Lemma~\ref{lem3.1}, all the cut edges of $G$ are pendant edges. By Lemma~\ref{lem3.2}, all the cycles of $G$ are end-blocks. So $G$ must be a graph obtained from a bundle of $k$ cycles by attaching  pendant edges to the common vertex of the $k$ cycles.  By Lemmas~\ref{lem3.3} and~\ref{lem3.4}, the length of any cycle of $G$ is $3$. Hence $G \cong C_{0}(n, k)$. By Lemma \ref{lem2.2}, we have
\begin{eqnarray*}
Sz_e(C_{0}(n, k))&=&kf(n+k-4,3)
=2k(n+k-4)+3k
=2kn+2k^{2}-5k,\\
2Sz_{ev}(C_{0}(n, k))&=&kg(n,n+k-4,3)+(n+k-2)(n-2k-1)\\ &=&2k(2n+k-4)+(n+k-2)(n-2k-1)\\&=&n^{2}-3n+3kn-5k+2.
\end{eqnarray*}
The proof is completed.\hspace{11.8cm}$\square$

\section{Conclusions}

In this paper, the edge Szeged index and edge-vertex Szeged index on cactus are discussed. The extremal graphs with minimum edge Szeged index and edge-vertex Szeged index in the class of all cacti with $n$ vertices and given number of cycles are obtained. For further study, it would be interesting to determine the extremal graph that has the maximum edge Szeged and edge-vertex Szeged index in these class of cacti. Moreover, it would be meaningful to study the edge szeged index and edge-vertex Szeged index of other kinds of graphs.

\section*{Acknowledgments}

This work was supported by the National Natural Science Foundation of China
(Nos. 11371052,11\\371193), the Fundamental Research Funds for the Central Universities (Nos. 2016JBM071, 2016JBZ012) and the $111$ Project of China (B16002).



\begin{thebibliography}{99}
\vspace{-7pt}\bibitem{Al.M.J} M. Alaeiyan, J. Asadpour, The vertex-edge Szeged index of bridge graphs, World Applied Sciences Journal. 14 (8) (2011) 1254--1257.

\vspace{-7pt}\bibitem{AF.T} T. Al-Fozan, P. Manuel, I. Rajasingh, R.S. Rajan, Computing Szeged index of certain nanosheets using partition
technique, MATCH Commun. Math. Comput. Chem. 72 (2014) 339--353.

\vspace{-7pt}\bibitem{Ao.M} M. Aouchiche, P. Hansen, On a conjecture about the Szeged index, Eur. J. Comb. 31 (2010) 1662--1666.







\vspace{-7pt}\bibitem{Cai.X} X. Cai, B. Zhou, Edge Szeged index of unicyclic graphs, MATCH Commun. Math. Comput. Chem. 63 (2010) 133--144.

\vspace{-7pt}\bibitem{C.LXL} L. Chen, X. Li, M. Liu, The (revised) Szeged index and the Wiener index of a nonbipartite graph, Eur. J. Comb. 36 (2014a) 237--246.

\vspace{-7pt}\bibitem{Deng.H} H. Deng, The trees on $n \geq 9$ vertices with the first to seventeenth greatest wiener indices are chemical trees, MATCH Commun. Math. Comput. Chem. 57 (2007) 393--402.

\vspace{-7pt}\bibitem{Do.A} A. Dobrynin, Graphs having the maximal value of the Szeged index, Croat. Chem. Acta. 70 (1997) 819--825.

\vspace{-7pt}\bibitem{Do.RI} A. Dobrynin, R. Entringer, I. Gutman, Wiener index of trees: theory and applications, Acta Appl. Math. 66 (2001) 211--249.

\vspace{-7pt}\bibitem{Dong.H} H. Dong, B. Zhou, Maximum wiener index of unicyclic graphs with fixed maximum degree, Ars Comb. 103 (2012) 407--416.

\vspace{-7pt}\bibitem{Fag.M} M. Faghani, A.R. Ashrafi, Revised and edge revised Szeged indices of graphs, Ars Math. Contemp. 7 (1) (2014) 153--160.

\vspace{-7pt}\bibitem{Gut.A} I. Gutman, A formula for the wiener number of trees and its extension to graphs containing cycles, Graph Theory Notes N. Y. 27 (1994) 9--15.

\vspace{-7pt}\bibitem{Gut.A.R} I. Gutman, A.R. Ashrafi, The edge version of the Szeged index, Croat. Chem. Acta. 81 (2008) 263--266.



\vspace{-7pt}\bibitem{guts} I. Gutman, S.K. Zar, B. Mohar, Fifty years of the wiener index, MATCH Commun. Math. Comput. Chem. 35 (1997) 1--259.

\vspace{-7pt}\bibitem{Kh.P} P. Khadikar, P. Kale, N. Deshpande, S. Karmarkar, V. Agrawal, Szeged indices
of hexagonal chains, MATCH Commun. Math. Comput. Chem. 43 (2000) 7--15.

\vspace{-7pt}\bibitem{P.V.K} P. Khadikar, S. Karmarkar, V.K. Agrawal, J. Singh, A. Shrivastava, I.
Lukovits, M.V. Diudea, Szeged index-applications for drug modeling, Lett. Drug Des. Disc. 2 (2005) 606--624.

\vspace{-7pt}\bibitem{Kha.M.H} M.H. Khalifeh, H. Yousefi-Azari, A.R. Ashrafi, A matrix method for computing Szeged and vertex PI indices of join and composition of graphs, Linear Algebra Appl. 429 (2008) 2702--2709.


\vspace{-7pt}\bibitem{Kh.M.H} M.H. Khalifeh, H. Yousefi-Azari, A.R. Ashrafi, Vertex and edge PI indices of Cartesian product graphs, Discrete
Appl. Math. 156 (2008) 1780--1789.

\vspace{-7pt}\bibitem{Lei.H}  H. Lei, H. Yang, Bounds for the sum-Balaban index and (revised) Szeged index of regular graphs, Appl. Math. Comput. 268 (2015) 1259--1266.

\vspace{-7pt}\bibitem{LinZ} H. Lin, B. Zhou, On the distance Laplacian spectral radius of graphs, Linear Algebra Appl. 475 (2015) 265--275.


\vspace{-7pt}\bibitem{LXue.M} X. Li, M. Liu, Bicyclic graphs with maximal revised Szeged index, Discrete Appl. Math. 161 (2013) 2527--2531.







\vspace{-7pt}\bibitem{Wang} S. Wang, On extremal cacti with respect to the Szeged index, Appl. Math. Comput. 309 (2017) 85--92.





\vspace{-7pt}\bibitem{zhang.H} H. Zhang, S. Li, L. Zhao, On the further relation between the (revised) Szeged index and the wiener index of graphs, Discrete Appl. Math. 206 (2016) 152--164.

\vspace{-7pt}\bibitem{Zhang.X} X. Zhang, Y. Liu, M. Han, Maximum Wiener index of trees with given degree sequence, MATCH Commun. Math. Comput. Chem. 64 (2010)  661--682.

\vspace{-7pt}\bibitem{ZhouB.X} B. Zhou, X. Cai, Z. Du, On Szeged indices of unicyclic graphs, MATCH Commun. Math. Comput. Chem. 63 (2010) 113--132.








\end{thebibliography}
\end{document}